\documentclass[reqno]{amsart}
\usepackage{diagrams}
\usepackage{amssymb}
\usepackage{mathrsfs}
\usepackage{cite}
\usepackage{tikz-cd}
\usepackage{MnSymbol}
\usepackage{a4wide}
 
\usepackage{xcolor}

\DeclareMathOperator\Q{\mathbb Q}
\DeclareMathOperator\R{\mathbb R}
\DeclareMathOperator\Z{\mathbb Z}

\newcommand{\Om}{\Omega}

\newcommand{\G}{{\mathbb G}}

\newtheorem{theorem}{Theorem}[section]
\newtheorem{lemma}[theorem]{Lemma}
\newtheorem{cor}[theorem]{Corollary}
\newtheorem{prop}[theorem]{Proposition}
\theoremstyle{definition}
\newtheorem{definition}[theorem]{Definition}

\theoremstyle{remark}
\newtheorem{remark}[theorem]{Remark}

\newcommand{\dontprint}[1]\relax

\newcommand{\be}{{\mathbf e}}
\newcommand{\bq}{{\mathbf q}}
\newcommand{\Coh}{\operatorname{Coh}}

\newcommand{\La}{\Lambda}

\renewcommand{\P}{{\mathbb P}}
\newcommand{\A}{{\mathbb A}}
\newcommand{\wt}{\widetilde}
\newcommand{\ot}{\otimes}

\newcommand{\Hom}{\operatorname{Hom}}

\newcommand{\Ext}{\operatorname{Ext}}

\newcommand{\CC}{{\mathcal C}}

\newcommand{\OO}{{\mathcal O}}
\newcommand{\UU}{{\mathcal U}}

\newcommand{\si}{\sigma}
\newcommand{\de}{\delta}
\newcommand{\sub}{\subset}

\newcommand{\Spec}{\operatorname{Spec}}

\newcommand{\Nm}{\operatorname{Nm}}

\newcommand{\lan}{\langle}
\newcommand{\ran}{\rangle}
\newcommand{\ov}{\overline}

\newcommand{\om}{\omega}
\newcommand{\la}{\lambda}
\renewcommand{\a}{\alpha}

\newcommand{\rk}{\operatorname{rk}}

\newcommand{\eps}{\epsilon}

\newcommand{\Pic}{\operatorname{Pic}}

\newcommand{\PGL}{\operatorname{PGL}}
\newcommand{\PSL}{\operatorname{PSL}}

\newcommand{\sspan}{\operatorname{span}}
\newcommand{\vareps}{\varepsilon}

\numberwithin{equation}{section}
\title{Exceptional pairs on del Pezzo surfaces and spaces of compatible Feigin-Odesskii brackets}
\author{Alexander Polishchuk}

\address{
    Department of Mathematics, 
    University of Oregon, 
    Eugene, OR 97403, USA; and National Research University Higher School of Economics, Moscow, Russia
  }
  \email{apolish@uoregon.edu}

\author{Eric Rains}
\address{Department of Mathematics, California Institute of Technology, Pasadena CA 91125}
\email{rains@caltech.edu}

\begin{document}

\begin{abstract} 
We prove that for every relatively prime pair of integers $(d,r)$ with $r>0$, there exists an exceptional pair $(\OO,V)$ on any del Pezzo surface of degree $4$,
such that $V$ is a bundle of rank $r$ and degree $d$. As an application, we prove that every Feigin-Odesskii Poisson bracket on a projective space can be included into
a $5$-dimensional linear space of compatible Poisson brackets. We also construct new examples of linear spaces of compatible Feigin-Odesskii Poisson brackets
of dimension $>5$, coming from del Pezzo surfaces of degree $>4$.
\end{abstract}

\maketitle

\section{Introduction}

We work over an algebraically closed ground field of arbitrary characteristic. 

Recall that a del Pezzo surface is a smooth projective surface $X$ with ample anticanonical divisor. For a divisor class $D$ on $X$ we set
$\deg(D)=D\cdot (-K)$, and for a vector bundle $V$ on $X$ we set $\deg(V)=\deg(c_1(V))$. The degree of $X$ is $\deg(-K)=(-K)^2$.

\medskip

\noindent
{\bf Theorem A}. {\it For any del Pezzo surface $X$ of degree $4$ and any pair of relative prime numbers $(d,r)$, where $r>0$, there
exists an exceptional pair $(\OO_X,V)$ on $X$ with $V$ an exceptional bundle of rank $r$ and $\deg(V)=d$.}

\medskip

In fact, we present in Theorem \ref{main-thm} a construction of a family of pairs $(\OO_X,V)$ for given $(d,r)=(\deg(V),\rk(V))$, depending on a way to present a degree $4$ del Pezzo surface
$X$ as a blow up of $\P^1\times\P^1$.
Let us say that an exceptional pair $(\OO_X,V)$ is {\it sporadic} if it is not of the form considered in Theorem \ref{main-thm} (for any way of presenting $X$ as a blow up of $\P^1\times\P^1$). We prove that sporadic pairs appear only when $|d|$ is not too large compare to $r$ (with an explicit quadratic bound).
A more precise statement is given in Theorem \ref{sporadic-thm}. 

For the proof of both Theorem A and Theorem \ref{sporadic-thm} we use crucially the action of the 
Weyl group of type $D_5$ on the Picard group of a degree $4$ del Pezzo surface $X$ (see \ref{blowdown-sec}). We also use, following a referee's suggestion\footnote{Our original proof used a complicated inductive procedure, involving del Pezzo surfaces of degree $1$ and some computations in the $E_8$ lattice.}, fully faithful embeddings of
the derived categories of certain weighted projective lines $\CC_4$ and $\CC_5$ into $D^b(\Coh X)$ (see Sec.\ \ref{C5-conn-sec} and \ref{C4-conn-sec}) and some results on
exceptional bundles over weighted projective lines (see Sec.\ \ref{weighted-sec}). Here $\CC_n$ denotes a stacky projective line with $n$ points of weight $2$ (i.e.,
having $\mu_2$ as automorphism groups). Our Theorem \ref{sporadic-thm} translates into a statement that in a certain range of rank and degree
every exceptional bundle on $\CC_5$ comes from $D^b(\Coh \CC_4)$ contained in $D^b(\Coh \CC_5)$ as the right orthogonal to one of the exceptional sheaves supported at stacky points
on $\CC_5$ (see Proposition \ref{C5-omnipresent-prop}).

For del Pezzo surfaces of degree $k\ge 5$, we show that the possible pairs $(d=\deg(V),r=\rk(V))$ for an exceptional pair $(\OO,V)$ have to satisfy inequality
\begin{equation}\label{d2-krd-ineq}
d^2-krd+kr^2\ge -k,
\end{equation}
with equality for $k=9$.
In particular, for $k\ge 5$ not all relatively prime pairs occur (see Proposition \ref{qu-ineq-prop}). We also prove in Proposition \ref{higher-deg-dP-prop} realizability in the case 
$5\le k\le 8$ (resp., $k=9$) of pairs satisfying
$$-k\le d^2-krd+kr^2\le -1,$$
(resp., $d^2-9rd+9r^2=-9$) with some caveats for $k=8$. Note that in the case $k=9$ our equality implies that $3|d$. 

Our interest in exceptional pairs of the form $(\OO,V)$ on del Pezzo surfaces is due to their relevance in the theory of Feigin-Odesskii Poisson brackets on projective spaces.
Recall that for a simple vector bundle $E$ of rank $r$ and degree $d>0$ on an elliptic curve $C$ (or its degeneration), one has a natural Poisson bracket on $\P H^0(C,E)^*$,
constructed in terms of the geometry of vector bundles on $C$ (see \cite{FO}, \cite{P-ell}). This Poisson bracket (well defined up to proportionality) 
depends only on the elliptic curve $C$ and the discrete invariants $(d,r)$ of $E$. We will refer to it as {\it FO bracket} and denote it by $q_{d,r}(C)$.

It was observed in \cite{HP-bih} (by generalizing earlier construction of Odesskii-Wolf \cite{OW})
that if $(\OO_X,V)$ is an exceptional pair on a del Pezzo surface $X$ then for every anticanonical divisor $C\sub X$ the restriction map
$H^0(X,V)\to H^0(C,V|_C)$ is an isomorphism, and the Poisson brackets on $\P H^0(X,V)^*\simeq \P^{d-1}$ (where $d=h^0(V)=\deg(V)$) 
coming from $V|_C$ for various anticanonical divisors form a vector
space of compatible Poisson brackets (which means that the Schouten bracket between each pair of the corresponding bivectors is zero).
In the case $r=1$ this gives $9$-dimensional linear subspaces of compatible FO-brackets on $\P^{d-1}$ containing any given FO-bracket of type $(d,1)$. 
For $r>1$, \cite{HP-bih} gives some partial results on compatible FO-brackets of type $(d,r)$.

\medskip

\noindent
{\bf Theorem B}. {\it (i) For any elliptic curve $C$ and any relatively prime positive numbers $(d,r)$, with $d>r+1$, the FO bracket $q_{d,r}(C)$ on $\P^{d-1}$ is contained in a $5$-dimensional 
linear subspace of compatible FO brackets on $\P^{d-1}$ of type $(d,r)$. If $r=2$ or $r=d-2$ then $q_{d,r}(C)$ is contained in a $6$-dimensional linear 
subspace of compatible FO brackets on $\P^{d-1}$ of type $(d,r)$.

\noindent (ii) For $5\le k\le 9$ and a relatively prime pair $(d,r)$ such that $d>r+1$ and 
$$-k\le d^2-krd+kr^2\le -1,$$
where $3|d$ in the case $k=9$, any FO-bracket $q_{d,r}(C)$ on $\P^{d-1}$ is contained in a $(k+1)$-dimensional
linear subspace of compatible FO brackets on $\P^{d-1}$ of type $(d,r)$.

\noindent (iii) Any linear subspace of compatible Poisson brackets on $\P^{d-1}$, whose generic element is a nonzero FO bracket of type $(d,r)$, lifts uniquely
to a linear subspace of compatible Poisson bracket on the affine space $\A^d$.
}

\medskip

The proof uses results of \cite{HP-bih}, our study of exceptional pairs $(\OO,V)$ on del Pezzo surfaces, and the results on existence of
non-isotrivial anticanonical pencils passing through a given anticanonical curve in a del Pezzo surfaces, established in Sec.\ \ref{pencils-sec}
(some of them are most likely known to the experts).

More precisely, given an exceptional pair $(\OO,V)$ on a del Pezzo surface of degree $k$, with $V$ of degree $d$ and rank $r$, the construction of \cite{HP-bih} provides
$k+1$ compatible FO-brackets on $\P^{d-1}$ of type $(d,r)$, and gives a criterion for their linear independence. Theorem A constructs such an exceptional pair for every slope $d/r$ 
on a degree $4$ del Pezzo surface. Although some slopes are achievable on higher degree del Pezzo surfaces, $k=4$ turns out to be the largest degree for which {\em all} slopes are achievable.  Moreover, even if we had only been interested in surfaces of lower degree, it would still be easier to work with degree 4, as for ``most" slopes (all but finitely many of any given rank) there is an essentially unique pair of that slope; this fails in lower degree.  This criticality of degree 4 appears in other ways; see Remark \ref{crit-rem}.

Note that Theorem B does not give the maximal number of compatible FO-brackets for all pairs $(d,r)$. For example, there are well known Poisson isomorphism 
$q_{d,r}(C)\simeq q_{d,r'}(C)$ for $rr'\equiv 1\mod(d)$ (see \cite{P-ell}), and in this way we see that $q_{7,3}(C)\simeq q_{7,5}(C)$ and $q_{7,4}(C)\simeq q_{7,2}(C)$ can be included into $6$-dimensional spaces of compatible FO brackets, which cannot be seen directly from Theorem B.

\bigskip

\noindent
{\it Acknowledgments}. We are grateful to the anonymous referee for many useful comments and especially for suggesting a more conceptual approach to studying exceptional pairs
$(V,\OO)$ on a del Pezzo surface of degree $4$, using the connection with weighted projective lines.
This material is based upon work supported by the National Science Foundation under grant No. DMS-1928930 and by the Alfred P. Sloan Foundation under grant G-2021-16778, while both authors were in residence at the Simons Laufer Mathematical Sciences Institute (formerly MSRI) in Berkeley, California, during the Spring 2024 semester. In addition, A.P. is partially supported by the NSF grant DMS-2349388, by the Simons Travel grant MPS-TSM-00002745,
and within the framework of the HSE University Basic Research Program. A.P. is also grateful to the IHES, where part of this work was done, for hospitality and excellent working 
conditions.

\section{Preliminaries}

\subsection{Exceptional objects on del Pezzo surfaces}

Let $X$ be a del Pezzo surface, and let $Q=-K$ denote the anticanonical divisor on $X$. For an exceptional bundle $V$, its degree $\deg(V)=c_1(V)\cdot Q$ is
relatively prime to the rank (since the restriction to a smooth anticanonical divisor is a simple bundle on an elliptic curve, by \cite[Lem.\ 3.6]{KO}).
It is known that an exceptional bundle on $X$ is slope stable with respect to $Q$ (see \cite{Gor}) and is uniquely determined by its class in $K_0$: the latter fact follows from stability (see \cite[Prop.\ 5.14]{Rains}).

We will use the Hirzebruch-Riemann-Roch theorem on $X$ in the form
$$\chi(V,V')=-\rk(V)\rk(V')+\rk(V)\chi(V')+\chi(V)\rk(V')-c_1(V)\cdot (c_1(V')+\rk(V')Q)
$$
(see \cite[(4.5)]{Rains}). 
For an exceptional object $V$ we have $\chi(V,V)=1$, so for such $V$,
\begin{equation}\label{chi-V-eq}
\chi(V)=\frac{r^2+c_1(V)\cdot (c_1(V)+rQ)+1}{2r},
\end{equation}
where $r=\rk(V)$.

\begin{lemma}\label{dP-lem}
(i) Let $F$ be an exceptional object in $D^b(\Coh X)$. Then either $F\simeq V[n]$, where $V$ is an exceptional bundle, or
$F\simeq \OO_R(m)[n]$, where $R$ is a $(-1)$-curve.


\noindent
(ii) Let $\pi:X\to X'$ be the blow down of a collection of non-intersecting $(-1)$-curves, and let $V$ be an exceptional vector bundle
on $X$ such that $c_1(V)\in \pi^*\Pic(X')$. Then $V\simeq \pi^*W$ for an exceptional vector bundle on $X'$.
\end{lemma}

\begin{proof}
(i) See \cite[Prop.\ 2.9, 2.10]{KO}.


\noindent
(ii) See \cite[Cor.\ 3.2]{KO}.
\end{proof}

\subsection{Action of the Weyl group on geometric markings}\label{blowdown-sec}

Let $X$ be a del Pezzo surface of degree $4$.
Given a representation of $X$ as an iterated blow up of $\P^2$, we get an isomorphism of lattices
\begin{equation}\label{phi-lattice-isom}
\phi:I^{1,5}\to \Pic(X),
\end{equation}
where $\Pic(X)$ is equipped with the intersection form, and $I^{1,5}=\bigoplus_{i=0}^5 \Z \be_i$ with $\be_0\cdot \be_0=1$, $\be_i\cdot \be_i=-1$ for $i>0$,
$\be_i\cdot \be_j=0$ for $i\neq j$ (see \cite[Ch.\ 8]{Dolg}). Namely, $\phi(\be_0)=h$, the pull-back of $\OO_{\P^2}(1)$ on $\P^2$, while for $i=1,\ldots,5$, 
$\phi(\be_i)=e_i$, the classes of exceptional curves.
Such an isomorphism $\phi$ is called a {\it geometric marking} of $X$, and the corresponding basis $(h,e_1,\ldots,e_5)$ of $\Pic(X)$ a {\it geometric basis}.

One has $\phi(\bq)=Q=-K_X$, where $\bq=3\be_0-\sum_{i=1}^5 \be_i$, and the sublattice $\lan \bq\ran^\perp\sub I^{1,5}$, with the form $(v,w)=-v\cdot w$,
is identified with the $D_5$ root lattice. More precisely, we choose this identification so that the simple roots are sent by $\phi$ to
$$\a_1\mapsto e_1-e_2, \ \ \a_2\mapsto e_2-e_3, \ \ \a_3\mapsto e_3-e_4, \ \ \a_4\mapsto e_4-e_5, \ \ \a_5\mapsto h-e_1-e_2-e_3.$$ 
Furthermore, the group of isometries of $I^{1,5}$ fixing $\bq$ is isomorphic to the Weyl group $W(D_5)$ with its natural action on the root lattice $\lan \bq\ran^\perp$
(see \cite[Sec.\ 8.2]{Dolg}).
For computations with the action of $W(D_5)$ it is convenient to use the following orthonormal basis of $\lan Q\ran^\perp\sub \Pic(X)_{\Q}$
with respect to the negative of the intersection form (over $\Q$):
\begin{equation}\label{eps-i-geom-eq}
\eps_i=\frac{1}{2}(h-\sum_{i=1}^5 e_i)+e_i, \ \ i=1,\ldots,5,
\end{equation}
so that $\phi(\a_i)=\eps_i-\eps_{i+1}$ for $i=1,\ldots,4$ and $\phi(\a_5)=\eps_4+\eps_5$.
 The action of $W(D_5)$ is generated by permutations of $(\eps_i)$, as well as by the sign changes of even number of $\eps_i$.
 One has
 $$h=\frac{3}{4}Q-\frac{1}{2}\sum_{i=1}^5 \eps_i, \ \ e_j=\frac{Q}{4}+e_j-\frac{1}{2}\sum_{i=1}^5\eps_i, \ j=1,\ldots,5.$$
 
Thus, we have a natural right action of the group $W(D_5)$ on the set of lattice isomorphisms \eqref{phi-lattice-isom}: the action of $w\in W(D_5)$ sends 
$\phi$ to $\phi\circ w$.
An important fact we need is that this $W(D_5)$-action preserves the set of geometric markings (see \cite[Sec.\ 8.2.8]{Dolg}). Equivalently, the natural action of the 
group of isometries of $\Pic(X)$ fixing $Q$ (isomorphic to $W(D_5)$) on bases of $\Pic(X)$ sends a geometric basis to a geometric basis.
 
We can also represent $X$ as a blow up of $\P^1\times \P^1$ at $4$ points, which gives a basis of $\Pic(X)$ of the form
\begin{equation}\label{blowdown-structure}
s,f,E_1,E_2,E_3,E_4
\end{equation}
where $E_i$ are classes of exceptional divisors (we use upper case to distinghuish from the basis obtained from the blow up of $\P^2$), 
and $s$ and $f$ are the classes of the two rulings on $\P^1\times\P^1$, so that
$s\cdot f=1$,  $E_i^2 = -1$,
and all other intersections are zero. 
We will refer to such bases \eqref{blowdown-structure} as {\it $\P^1\times\P^1$-type geometric bases}.

\begin{lemma} The action of the group of isometries of $\Pic(X)$ fixing $Q$ sends a $\P^1\times\P^1$-type geometric basis to
a $\P^1\times\P^1$-type geometric basis.
\end{lemma}

\begin{proof}
Let $\a:\Pic(X)\to \Pic(X)$ be such an isometry. Then $E'_1=\a(E_1),\ldots,E'_4=\a(E_4)$ are classes of disjoint $(-1)$-curves (see \cite[Sec.\ 8.2.6]{Dolg}), therefore, they can be contracted.
Let $X'$ denote the corresponding blown down surface. The classes $s'=\a(s)$ and $f'=\a(f)$ form a basis of $\Pic(X')$, so that $2s'+2f'=Q-E'_1-\ldots-E'_4$ is the anticanonical
class of $X'$. Since $(s')^2=(f')^2=0$ and $s'\cdot f'=1$, it follows that $X'\simeq \P^1\times \P^1$, and $s'$ and $f'$ are the classes of the rulings.
\end{proof}

We associate a $\P^1\times\P^1$-type geometric basis $(s,f,E_1,\ldots,E_4)$ with a geometric basis $(h,e_1,\ldots,e_5)$ by setting
\begin{equation}\label{s-f-marking-eq}
s=h-e_4, \ \ f=h-e_5, \ \ E_1=e_1, \ \ E_2=e_2, \ \ E_3=e_3, \ \ E_4=h-e_4-e_5
\end{equation}

\subsection{Transformations of exceptional pairs $(V,\OO)$}







Let $X$ be a del Pezzo surface of degree $k$.
We have the following three natural operation on exceptional pairs $(V_1,V_2)$ on $X$: 

(1) duality $(V_1,V_2)\mapsto (V_2^\vee, V_1^\vee)$;

(2) mutation $(V_1,V_2)\mapsto (L_{V_1}V_2,V_1)$, where the {\it left mutation} $L_{V_1}V_2$ of $V_2$ fits into an exact triangle
$$L_{V_1}V_2\to R\Hom(V_1,V_2)\ot V_1\to V_2\ldots;$$

(3) rotation $(V_1,V_2)\mapsto (V_2(-Q),V_1)$.

Combining these operations we get the following two operations on exceptional pairs of the form $(V,\OO)$:
\begin{align}\label{mu-transf-eq}
&(V,\OO)\mapsto (M(V):=L_{\OO}(V^\vee),\OO), \nonumber \\
&(V,\OO)\mapsto (R(V):=V^\vee(-Q),\OO).
\end{align}

We can calculate the effect of these operations on the slope of $V$ (and on other invariants of $V$).

\begin{lemma}\label{M-R-lem}
One has
\begin{align*}
&\mu(M(V))=-\frac{\mu}{\mu+1},\\
&\mu(R(V))= -k-\mu,
\end{align*}
where $\mu=\mu(V)$.
Furthermore, 
$$\rk(M(V))=-\deg(V)-\rk(V), \ \ \rk(R(V))=\rk(V),$$
$$c_1(M(V))=c_1(V), \ \ c_1(R(V))=-c_1(V)-\rk(V)Q.$$
\end{lemma}


\begin{remark}\label{crit-rem}
Note that both $M$ and $R$ are involutions acting by fractional-linear transformation on the slope $\mu$, and thus generate a 
dihedral subgroup of $\PGL_2(\R)$ (infinite or finite).  
The generating rotation (i.e., the action of $RM$ in $\PSL_2(\R)$) is elliptic for $k<4$, hyperbolic for $k>4$, and parabolic for $k=4$.  For $k>4$, the open interval between the two fixed points of $RM$ is given by the inequality $\mu^2+k\mu+k<0$,
and thus there are only finitely many orbits of achievable slopes in that interval (see Proposition \ref{higher-deg-dP-prop} below).
\end{remark}

\section{Exceptional bundles on the weighted projective lines}\label{weighted-sec}

\subsection{Generalities on hyperelliptic weighted projective lines}

We refer to \cite{Lenzing} and \cite{Meltzer} for basic facts about weighted projective lines and exceptional sheaves on them. We will only consider {\it hyperelliptic}
weighted projective lines associated with a collection of $n$ distinct points $\la_1,\ldots,\la_n$ on $\P^1(k)$ (``hyperelliptic" means that all the weights are $2$).
Recall that for the normalization $\la_1=\infty$, $\la_2=0$, $\la_3=1$, the corresponding weighted projective line 
$\CC=\CC_n=\CC(\la_1,\ldots,\la_n)$ is defined as a stacky quotient $(\Spec(S(\la_\bullet))-0)/\G_m$, for the graded algebra
$$S(\la_\bullet)=k[X_1,\ldots,X_n]/(X_i^2-X_2^2+\la_i X_1^2 \ | \ i=3,\ldots,n).$$ 
The $\Z$-grading on the algebra $S(\la_\bullet)$ (where $\deg(X_i)=1$) extends to a grading with values in a bigger abelian group
$L$ with generators $x_i$ (so that $\deg(X_i)=x_i)$ and defining relations $2x_1=\ldots=2x_n$.

For every element $l\in L$, the coresponding free graded module $S(\la_\bullet)(l)$ induces a line bundle on $\CC$, and this defines
an isomorphism $L\simeq \Pic(\CC)$. We will denote by $\OO_{\CC}(1)$ the line bundle corresponding to the canonical element $2x_i\in L$.
The dualizing line bundle on $\CC$ is $\om_{\CC}=\OO_{\CC}(n-2)(-\sum_{i=1}^n x_i)$.
We define the degree of line bundles on $\CC$ using the homomorphism $\deg:L\to \Z$ sending each $x_i$ to $1$.

The category of coherent sheaves on $\CC$ is hereditary (has trivial $\Ext^{\ge 2}$), so every indecomposable object of the derived category
$D^b(\Coh \CC)$ is a sheaf up to a shift.
For every $i=1,\ldots,n$, we have two simple coherent sheaves on $\CC$, $S_i$ and $S'_i$, fitting into the exact sequences
$$0\to \OO_{\CC}(-x_i)\to \OO_{\CC}\to S_i\to 0,$$
$$0\to \OO_{\CC}\to \OO_{\CC}(x_i)\to S'_i\to 0.$$
We have a full exceptional collection
\begin{equation}\label{Cn-main-coll}
\OO_{\CC}(-1), \OO_{\CC}, (S_i)_{i=1,\ldots,n}
\end{equation}
 in $D^b(\Coh \CC)$ (where the last block consists of mutually orthogonal objects).

It is known that every exceptional vector bundle on $\CC$ is slope stable (see \cite[Prop.\ 2.3.7]{Meltzer}).
An exceptional sheaf is uniquely determined by its class in $K_0$ (see \cite[Lem.\ 3.4.1]{Meltzer}).
It is also known (see \cite[Thm.\ 3.5.1]{Meltzer}) that the degree and rank of an exceptional bundle on $\CC$ are coprime.

We denote by $Q(D_n)\sub P(D_n)$ the root and the weight lattices of type $D_n$, respectively.
We denote by $\eps_1,\ldots,\eps_n$ the standard orthonormal basis of $Q(D_n)_{\Q}$ (used in \cite{Bour}), so that
the roots are $\pm \eps_i\pm \eps_j$. The weights of $D_n$ are elements $\sum_i y_i\eps_i$ such that $y_i\in \Z/2$ and $y_i-y_j\in \Z$.

\begin{lemma}\label{hyperell-Dn-lem} 
(i) For $\CC=\CC(\la_1,\ldots,\la_n)$, let $\La_n\sub K_0(\CC)$ denote the kernel of the antisymmetrization of the Euler form $\chi(\cdot,\cdot)$.
Then $x\in \La_n$ if and only if $\deg(x)=\rk(x)=0$.
Furthermore, $\La_n$ equipped with the (symmetric) form $\chi(\cdot,\cdot)$
is isomorphic to $Q(D_n)$, the root lattice of type $D_n$, in such a way that 
\begin{equation}\label{eps-Cn-eq}
\eps_i=[S_i]+\frac{1}{2}([\OO_\CC(-1)]-[\OO_\CC]).
\end{equation}

\noindent
(ii) One has a decomposition $K_0(\CC)_{\Q}=M\oplus (\La_n)_{\Q}$, where $M$ is the left (or right) orthogonal to $\La_n$ with respect to $\chi$.
Let $\pi:K_0(\CC)_{\Q}\to (\La_n)_{\Q}\simeq Q(D_n)_{\Q}$ denote the projection along $M$. Then 
$$\pi[\OO_\CC]=\pi[\OO_\CC(-1)]=\frac{1}{2}\sum_{i=1}^n\eps_i, \ \ \pi[S_i]=\eps_i, \ i=1,\ldots,n.$$
In particular, $\pi(K_0(\CC))$ is contained in $P(D_n)$, the weight lattice of $D_n$.

\noindent
(iii) We have
\begin{equation}\label{Cn-chi-formula}
\chi([V],[V'])=(1-\frac{n}{4})r(V)r(V')+\frac{1}{2}(d(V')r(V)-d(V)r(V'))+\chi(\pi[V],\pi[V']).
\end{equation}
\end{lemma}

\begin{proof} (i) Since the classes of objects of the exceptional collection \eqref{Cn-main-coll} freely generate $K_0$, the Euler form $\chi$
is determined by the condition that $\chi(E,E)=1$ for each object $E$ of \eqref{Cn-main-coll} and by
$$\chi(\OO_\CC(-1),\OO_\CC)=2, \chi(\OO_\CC(-1),S_i)=\chi(\OO_\CC,S_i)=1,$$
with other values on objects of \eqref{Cn-main-coll} vanishing by semiorthogonality.
Thus, for $\chi^-(A,B)=\chi(A,B)-\chi(B,A)$ we have
\begin{align*}
&\chi^-(a_1[\OO_\CC(-1)]+a_0[\OO_\CC]+\sum_i b_i S_i,a'_1[\OO_\CC(-1)]+a'_0[\OO_\CC]+\sum_i b'_i S_i)\\
&=2(a_1a'_0-a'_1a_0)+\sum_i(a_0+a_1)b'_i-\sum_i (a'_0+a'_1)b_i.
\end{align*}
On the other hand, 
$$\rk(a_1[\OO_\CC(-1)]+a_0[\OO_\CC]+\sum_i b_i S_i)=a_0+a_1, \ \ \deg(a_1[\OO_\CC(-1)]+a_0[\OO_\CC]+\sum_i b_i S_i)=-2a_1+\sum_i b_i.$$
This immediately shows that the kernel $\La_n$ of $\chi^-(\cdot,\cdot)$ coincides with $\ker(\rk,\deg)\sub K_0(\CC)$. The above formula also shows that $\eps_i\in (\La_n)_{\Q}$
given by \eqref{eps-Cn-eq} are orthonormal with respect to $\chi(\cdot,\cdot)$. 

Furthermore, $\La_n=\ker(\rk,\deg)$ is generated by the classes $S_i-S_j$ and $[\OO_\CC(-1)]-[\OO_\CC]+2S_i$, which are given in terms of the basis $(\eps_i)$ as
$$S_i-S_j=\eps_i-\eps_j, \ \ [\OO_\CC(-1)]-[\OO_\CC]+2S_i=2\eps_i,$$
so $\La_n$ coincides with the root lattice $Q(D_n)$. 

\noindent
(ii) One immediately checks that the classes $[\OO_\CC]-\sum_i\eps_i/2$ and $[\OO_\CC(-1)]-\sum_i\eps_i/2$ are orthogonal to $\La_n$ with respect to $\chi(\cdot,\cdot)$.
This gives the required decomposition with $M$ spanned by these two classes and also implies the formulas for $\pi[\OO_\CC]$, $\pi[\OO_\CC(-1)]$ and $\pi[S_i]$.

\noindent
(iii) This formula can be checked on the generating classes $[\OO_\CC(-1)]$, $\OO_\CC$, $(\eps_i)$. 
\end{proof}



Let us now consider the case $n=4$.
This case is special in that the corresponding weighted projective lines have more derived autoequivalences (see \cite{LM}).
Also, in this case the first term in the formula \eqref{Cn-chi-formula} vanishes.

\begin{definition}\label{Brd-def}
For a pair of integers $(r,d)$, not both even, let us denote by $B(r,d)\sub \Q^4=Q(D_4)_{\Q}$ the set of vectors $y=y_1\eps_1+\ldots+y_4\eps_4$ such that 
\begin{itemize}
\item $\sum_{i=1}^4 y_i^2=1$;
\item $y_i\equiv r/2\mod \Z$;
\item $\sum_{i=1}^4 y_i\equiv d \mod 2\Z$;
\end{itemize}
\end{definition}

It is easy to see that $B(r,d)$ consists of exactly $8$ elements forming a single $W(D_4)$-orbit. 
More precisely, $B(r,d)$ depends only on the parity of $r$ and $d$, and 
$$B(0,1)=W(D_4)\cdot \eps_1,$$
$$B(1,0)=W(D_4)\cdot \frac{1}{2}(\eps_1+\eps_2+\eps_3+\eps_4),$$
$$B(1,1)=W(D_4)\cdot \frac{1}{2}(\eps_1+\eps_2+\eps_3-\eps_4).$$
There is a unique element of $B(r,d)$ in the fundamental chamber of $D_4$, $y_1\ge y_2\ge y_3\ge |y_4|$.

\begin{lemma}\label{C4-exc-lem}
Let $\CC=\CC(\la_1,\ldots,\la_4)$.
For every coprime $(r,d)$, a class $v\in K_0(\CC)$ with $\chi(v,v)=1$, $\rk(v)=r$ and $\deg(v)=d$ is
determined by $\pi(v)=\sum_{i=1}^4 y_i\eps_i$ which lies in $B(r,d)$.
For every coprime $(r,d)$ with $r>0$
(resp., for $(r,d)=(0,1)$), there are $8$ exceptional sheaves $F$ with $\rk(F)=r$ and $\deg(F)=d$, so that $\pi[F]$ realize all $8$ vectors in $B(r,d)$.
\end{lemma} 

\begin{proof}
From formula \eqref{Cn-chi-formula} with $n=4$, we immediately see that $\chi(v,v)=1$ if and only if $\sum_i y_i^2=1$.
Furthermore, since $\pi[S_i]=\eps_i$, by the same formula, $\chi(v,[S_i])=r/2+y_i$, which implies that $y_i\equiv r/2 \mod \Z$.
On the other hand, this formula gives $\chi([\OO_\CC],v)=d/2+\sum_i y_i/2$. Hence, $\sum_i y_i\equiv d \mod 2\Z$,
which shows that $\pi(v)$ is in $B(r,d)$.

To show the second statement we use the fact that for any slope $\mu$ there exists an autoequivalence 
of $D^b(\Coh \CC_4)$ taking indecomposable bundles
of slope $\mu$ to indecomposable bundles of slope $0$ (see \cite{LM}). Therefore, it is enough to check the second assertion
for line bundles of degree $0$, in which case it is straightforward.
\end{proof}

\subsection{Omnipresent exceptional bundles on $\CC_5$}

Recall (see \cite[Sec.\ 6.3]{Meltzer}) that an exceptional bundle $V$ on a weighted projective line $\CC_n$ is called {\it omnipresent} if $\Hom(V,S_i)\neq 0$ and $\Hom(V,S'_i)=0$ for all $i=1,\ldots,n$.
Since the corresponding $\Ext^1$-groups vanish, this is equivalent to the numerical condition $\chi(V,S_i)\neq 0$, $\chi(V,S'_i)\neq 0$.
By Serre duality, this is equivalent to $\chi(S_i,V)\neq 0$, $\chi(S'_i,V)\neq 0$ for all $i$.

Our goal is to show the absence of omnipresent bundles on $\CC_5=\CC(\la_1,\ldots,\la_5)$ in a certain range of $(\rk(V),\deg(V))$ (see Proposition \ref{C5-omnipresent-prop} below).

\begin{lemma}\label{C5-main-lem}
Let $r$ be a positive integer.  If $v=\sum_{i=1}^5 y_i\eps_i\in P(D_5)$ is in the fundamental chamber of $D_5$ (i.e., $y_1\ge \ldots\ge y_4\ge |y_5|$) and satisfies 
$|v|^2 = r^2/4+1$ and $\sum_i y_i\le r/2+\sqrt{r}-1$, then one of the following holds:
\begin{enumerate}
\item $r$ is odd and $v=\frac{r}{2}\eps_1+\frac{1}{2}(\eps_2+\eps_3+\eps_4)\pm \frac{1}{2}\eps_5$;
\item $r$ is even and $v=\frac{r}{2}\eps_1+\frac{1}{2}\eps_2$;
\item $r$ is an odd square and $v=(\frac{r}{2}-1)\eps_1+\frac{\sqrt{r}}{2}(\eps_2+\eps_3+\eps_4-\eps_5)$;
\item $r$ is an even square and $v=(\frac{r}{2}-1)\eps_1+\sqrt{r}\eps_2$;
\item $r=11$ and $v=\frac{5}{2}(\eps_1+\eps_2+\eps_3+\eps_4-\eps_5)$; 
\item $r=16$ and $v=4(\eps_1+\eps_2+\eps_3+\eps_4)-\eps_5$.
\end{enumerate}
\end{lemma}

\begin{proof}
Since $v\in P(D_5)$, we either have $y_i\in \Z$ for all $i$, in which case $|v|^2\in \Z$, so $r$ is even, or we have $y_i\in 1/2+\Z$ for all $i$, in which case
$|v|^2\in 1/4+\Z$, so $r$ is odd.

If $y_1>r/2$ then $y_1\ge r/2+1$, so that $|v|^2\ge y_1^2>r^2/4+1$, which is a contradiction.  We thus have $y_1\le r/2$.  Set $\delta:=r/2-y_1$.
As we have seen above, $\delta$ is a nonnegative integer. Our condition on the sum of $y_i$ can be rewritten as
$$y_2+\ldots+y_5\le \sqrt{r}-1+\de.$$
Furthermore, since $y_i\ge |y_5|$ for $i\le 4$, this implies the inequalities 
\[
|y_2+y_3+y_4+y_5|,\ |y_2+y_3-y_4-y_5|,\ |y_2-y_3+y_4-y_5|,\ |y_2-y_3-y_4+y_5|\le
\sqrt{r}-1+\delta
\]
Furthermore, since the left-hand sides are integers, we can replace $\sqrt{r}$ by $\lfloor \sqrt{r} \rfloor$.
The above inequalities define a hypercube with the side $\lfloor \sqrt{r} \rfloor-1+\delta$ in the Euclidean space $\R^4$ with coordinates $y_2,\ldots,y_5$.
On the other hand, we have
\[
|y_2|,\ |y_3|,\ |y_4|,\ |y_5|\le r/2-\delta,
\]
which is a hypercube with the side $r-2\delta$.
Since for a point in a hypercube in $\R^4$ with side $a$, the distance from the center is bounded by $a$, we deduce that
\[
y_2^2+y_3^2+y_4^2+y_5^2\le \min(\lfloor \sqrt{r} \rfloor-1+\delta,r-2\delta)^2,
\]
and thus
\begin{equation}\label{main-ineq}
r^2/4+1=|v|^2\le (r/2-\de)^2+\min(\lfloor \sqrt{r} \rfloor-1+\delta,r-2\delta)^2.
\end{equation}
Hence,
\[
r^2/4+1 \le \min(f(\de),g(\de)),
\]
where 
$$f(x)=(r/2-x)^2+(\sqrt{r}-1+x)^2, \ \ g(x)=(r/2-x)^2+(r-2x)^2=\frac{5}{4}(r-2x)^2.$$

The roots of the quadratic equation
$$f(x)=r^2/4+1$$
are $x=1$ and $x=r/2-\sqrt{r}$. Hence, we have
\begin{equation}\label{ineq1}
r^2/4+1>f(\de) \ \text{ for } 1<\delta<r/2-\sqrt{r}.
\end{equation}

Note also that $f(x_0)=g(x_0)$ for $x_0=(r-\sqrt{r}+1)/3$ and 
$$\min(f(x),g(x))=\begin{cases} f(x), & x\le x_0, \\ g(x), & x\ge x_0.\end{cases}$$

Assume $r>20$ and $\delta>1$. Then 
$$x_0=(r-\sqrt{r}+1)/3<r/2-\sqrt{r}.$$
In particular, by \eqref{ineq1}, $r^2/4+1>f(x_0)$.
If $\delta<x_0$ then using \eqref{ineq1} again we get 
$$r^2/4+1>f(\de)=\min(f(\de),g(\de)),$$
which is a contradiction. Hence, $\delta\ge x_0$. But the function $g(x)$ is decreasing, so we get
$$r^2/4+1>f(x_0)=g(x_0)\ge g(\de)=\min(f(\de),g(\de)),$$
which is again a contradiction.

It remains to enumerate the cases with $\delta=0$ and $\delta=1$, as well as
the finitely many cases with $r\le 19$ and $\delta>1$. 
Assume first that $r\le 19$ and $\delta>1$. As before, we only need to consider $\de\ge r/2-\sqrt{r}$ (otherwise, $r^2/4+1>f(\de)$).
Together with the inequality $r^2/4+1\le g(\de)$ this gives 
$$\frac{r}{2}-\sqrt{r}\le \de\le \frac{1}{2}\bigl(r-\sqrt{\frac{r^2+4}{5}}\bigr).$$
This gives the following possibilities: 1) $\de=5$, $r=18$; 2) $\de=4$, $r=16,15$; 3) $\de=3$, $r=13,12,11$; 4) $\de=2$, $r=10,9,8$.
Of these, only the pairs 
$$(r,\de)=(16,4), \ (11,3), \ (9,2)$$
satisfy the original inequality \eqref{main-ineq} involving the floor function.
In each case, we can easily
enumerate all lattice vectors in the relevant simplex, and find that there
are no solutions for $r=9$, while for $r=11$ and $r=16$ the only solutions are 
given in (5) and (6).

For $\delta=0$, we must have $y_2^2+y_3^2+y_4^2+y_5^2=1$, and one can easily check that this leads to
the solutions given in (1) and (2).

Finally, for $\delta=1$, we must have
\[
|y_2+y_3+y_4+y_5|,\ |y_2+y_3-y_4-y_5|,\ |y_2-y_3+y_4-y_5|,\ |y_2-y_3-y_4+y_5|\le
\sqrt{r}
\]
and
\[
y_2^2+y_3^2+y_4^2+y_5^2 = r.
\]
This forces $(y_2,y_3,y_4,y_5)$ to be a vertex of the hypercube, and thus,
$(y_2,\ldots,y_5)$ is either $(\sqrt{r},0,0,0)$ or $\sqrt{r}(1/2,1/2,1/2,-1/2)$, giving cases (3)
and (4) and finishing the classification.
\end{proof}

\begin{prop}\label{C5-omnipresent-prop}
Let $V$ be an omnipresent exceptional bundle of rank $r$ and degree $d<0$ on $\CC_5$, and suppose that $r>(-d+1)^2$.  Then $(r,d)=(11,-2)$, and
$\pi[V]$ is in the $W(D_5)$-orbit of $-\frac{5}{2}\sum_{i=1}^5\eps_i$.
\end{prop}

\begin{proof}
We will use the notations and the results of Lemma \ref{hyperell-Dn-lem}.
Let $v:=\pi[V]\in P(D_5)$. Using formula \eqref{Cn-chi-formula},
the condition $\chi(V,V)=1$ can be rewritten as 
$$|v|^2=r^2/4+1.$$ 

On the other hand, stability of $V$ implies that for any line bundle $L$ of degree 0 one has
$\Hom(L,V)=0$, hence, $\chi(L,V)\le 0$. 
Taking $L$ to be either $\OO$ or of the form $\OO(x_i-x_j)$ or $\OO(x_i+x_j-x_k-x_l)$ and
applying formula \eqref{Cn-chi-formula}, we get the inequality
$$\om\cdot v\le \frac{r}{4}-\frac{d}{2}$$
for every $\om$ in the $W(D_5)$-orbit of $\sum_{i=1}^5 \eps_i/2$.
Let $w\in W(D_5)$ be an element such that $v_0:=w\cdot v$ is in the fundamental chamber of $D_5$.
Then since the same inequalities hold for $v_0=\sum_i y_i\eps_i$, we get that
$$\sum_{i=1}^5 y_i\le \frac{r}{2}-d<\frac{r}{2}+\sqrt{r}-1.$$
Since the inequality here is strict, cases (3), (4), and (6) do not arise, and thus we are left with cases (1), (2), and (5).  In the first two cases, the fact that $y_1=r/2$ immediately translates to the existence of a point sheaf with $\chi(V,x)=0$, so only (5) remains, in which case we find that $r=11$ and
\[
y_1+y_2+y_3+y_4+y_5=15/2\le 11/2-d<9/2+\sqrt{11},
\]
forcing $-d=2$.
\end{proof}

\begin{remark}
The fact that an omnipresent exceptional bundle of rank $11$ and degree $-2$ on $\CC_5$ indeed exists (at least for some $\CC_5$) follows from 
Theorem \ref{sporadic-thm} below. Namely, it corresponds to the bundle $(p^*T_{\P^2}(1))(-Q)$ on a del Pezzo surface of degree $4$ presented as a blow up $p:X\to\P^2$,
using the equivalence of $\lan \OO_X\ran^\perp$ with $D^b(\Coh \CC_5)$ considered in Proposition \ref{C5-orth-prop}.
\end{remark}

\section{Exceptional pairs on del Pezzo surfaces of degree $4$}

\subsection{Connection with the derived category of $\CC_5$}\label{C5-conn-sec}

Let $X$ be a del Pezzo surface of degree $4$ with a geometric marking $h,e_1,\ldots,e_5$.
The homomorphism $\det$ identifies the subgroup $\ker(\rk)\sub K_0(X)$ with $\Pic(X)$,
hence, we can identify the sublattice $\lan Q\ran^\perp\sub \Pic(X)$ of elements of degree $0$ in $\Pic(X)$ (which is
identified with the root lattice of $D_5$) with the subgroup $\ker(\rk,\deg)\sub K_0(X)$.

\begin{prop}\label{C5-orth-prop}
There is an equivalence 
$$\Phi: D^b(\Coh \CC)\rTo{\sim}\lan \OO_X\ran^{\perp}$$
for a weighted projective line $\CC=\CC_5$, such that
$$\Phi(\OO_\CC(-3)(\sum x_i))=\OO_X(-2h+\sum e_i), \ \ \Phi(\OO_\CC(-2)(\sum x_i))=\OO_X(-h)[1],$$
$$\Phi(S'_i)=\OO_X(e_i-h)[1], \ \ \Phi(S_i)=\OO_X(-Q-e_i+h)[1], \ i=1,\ldots,5.$$
One has
\begin{equation}\label{deg-rk-C5-X-eq}
\rk_X(\Phi(V))=-\deg_\CC(V), \ \ \deg_X(\Phi(V))=\rk_\CC(V)+2\deg_\CC(V).
\end{equation}
Furthermore, the isomorphism induced by $\Phi$,
$$Q(D_5)\simeq \ker(\rk_\CC,\deg_\CC)\rTo{\Phi} \ker(\rk_X,\deg_X)\simeq \lan Q\ran^\perp$$
coincides with the one coming from the geometric marking (see Sec.\ \ref{blowdown-sec}).
\end{prop}

\begin{proof}
The full exceptional collection on $X$ (see \cite{Orlov}),
$$(\OO_{e_i}(-1)[-1])_{i=1,\ldots,5},\OO_X(-2h),\OO_X(-h),\OO_X)$$
gives a full exceptional collection in $\lan \OO_X\ran^\perp$ (by deleting the last object).
The exact sequence
\begin{equation}\label{1st-mutation-eq}
0\to \OO_X(-2h)\to \OO_X(-2h+\sum_i e_i)\to \bigoplus_i \OO_{e_i}(-1)\to 0
\end{equation}
shows that by left mutating $\OO_X(-2h)$ through all $\OO_{e_i}(-1)[-1]$, we get a full exceptional collection
$$\OO_X(-2h+\sum_i e_i)[-1], (\OO_{e_i}(-1)[-1])_{i=1,\ldots,5}, \OO_X(-h)$$
in $\lan \OO_X\ran^\perp$.
Now the exact sequences
$$0\to \OO_X(-h)\to \OO_X(e_i-h)\to \OO_{e_i}(-1)\to 0$$
show that by right mutating all $\OO_{e_i}(-1)[-1]$ through $\OO_X(-h)$, we get an exceptional collection
\begin{equation}\label{mutated-exc-coll}
\OO_X(-2h+\sum_i e_i)[-1], \OO_X(-h), (\OO_X(e_i-h))_{i=1,\ldots,5}.
\end{equation}
Let us show that this collection is strong and compute the corresponding algebra.

Applying the functors $\Ext^*(?,\OO_X(-h))$ to exact sequence \eqref{1st-mutation-eq}, we get that $\Ext^i(\OO_X(-2h+\sum_i e_i),\OO_X(-h))=0$ for $i\neq 1$,
and we get an exact sequence
$$0\to \Hom(\OO_X(-2h),\OO_X(-h))\to \Ext^1(\bigoplus_i \OO_{e_i}(-1),\OO(-h))\to \Ext^1(\OO_X(-2h+\sum_i e_i),\OO_X(-h))\to 0.$$
Consider the $1$-dimensional spaces $L_i:=\Ext^1(\OO_{e_i}(-1),\OO_X(-h))$ and let $\wt{V}=\bigoplus_{i=1}^5 L_i$. The above exact sequence gives an embedding of 
the $3$-dimensional space $W:=\Hom(\OO_X(-2h),\OO_X(-h))$ into
$\wt{V}$, such that
$\Ext^1(\OO_X(-2h+\sum_i e_i),\OO_X(-h))$ is identified with the $2$-dimensional quotient $V:=\wt{V}/W$.
Furthermore, since the points that we blow up in $\P^2$ are not collinear, we have  $\Hom(\OO_X(-2h+e_i+e_j+e_k),\OO_X(-h))=0$ for any distinct triple $(i,j,k)$.
Hence, applying $\Ext^*(?,\OO_X(-h))$ to an analog of the exact sequence \eqref{1st-mutation-eq}, with only three classes $(e_i,e_j,e_k)$ instead of all five, shows that the natural projection $W\to L_i\oplus L_j\oplus L_k$ is an isomorphism. It follows that the five lines $(W+L_i)/W$ in $V=\wt{V}/W$ are distinct. 

Next, applying $\Ext^*(\OO_X(-2h+\sum_j e_j),?)$ to the exact sequence 
$$0\to \OO_X(-h)\to \OO_X(e_i-h)\to \OO_X(e_i)|_{e_i}\to 0$$
we see that $\Ext^i(\OO_X(-2h+\sum_j e_j),\OO_X(e_i-h))=0$ for $i\neq 1$,
while the natural map $\Ext^1(\OO_X(-2h+\sum_j e_j),\OO_X(-h))\to \Ext^1(\OO_X(-2h+\sum_j e_j),\OO_X(e_i-h))$ is surjective, so it can be identified with 
the natural projection $V=\wt{V}/W\to \wt{V}/(W+L_i)$.

We can view the projections $V\to \wt{V}/(W+L_i)$ as five distinct points on the projective line $\P(V^*)$, so we get an identification of the algebra of the collection
\eqref{mutated-exc-coll} with that of the strong full exceptional collection
\begin{equation}\label{C5-stand-exc-coll}
\OO_\CC(-1), \OO_\CC, (S_i)_{i=1,\ldots,5}
\end{equation}
on the corresponding weighted projective line $\CC=\CC_5$. Hence, there exists an equivalence 
$$\wt{\Phi}:D^b(\Coh \CC)\rTo{\sim} \lan\OO_X\ran^\perp,$$
identifying the exceptional collections \eqref{C5-stand-exc-coll} and \eqref{mutated-exc-coll}.
It is easy to see that the left mutation of $S_i$ through $\lan\OO_\CC(-1),\OO_\CC\ran$
is isomorphic to $S'_i[-1]$.
On the other hand, the left mutation of $\OO_X(e_i-h)$ through $\OO_X(-h)$ is $\OO_{e_i}(-1)[-1]$, and the exact sequence
$$0\to \OO_X(-Q+h-e_i)\to \OO_X(-2h+\sum_j e_j)\to \OO_{e_i}(-1)\to 0$$
shows that the left mutation of $\OO_{e_i}(-1)[-1]$ through $\OO_X(-2h+\sum_j e_j)[-1]$ is $\OO_X(-Q+h-e_i)[-1]$.
Hence, $\wt{\Phi}(S'_i)\simeq \OO_X(-Q+h-e_i)$.

To obtain the equivalence $\Phi$ we compose $\wt{\Phi}$ with the autoequivalence
$$F\mapsto F\ot \OO_\CC(2)(-\sum x_i)[1]$$
of $D^b(\Coh \CC)$.
Then $\Phi$ identifies the exceptional collection \eqref{mutated-exc-coll} in $\lan\OO_X\ran^\perp$ with
the collection
$$\OO_\CC(-3)(\sum x_i), \OO_\CC(-2)(\sum x_i), (S'_i)_{i=1,\ldots,5}$$
on $\CC$, and sends $S_i$ to $\OO_X(-Q+h-e_i)[1]$. 

Since the classes of the objects of our exceptional collections generate the corresponding Grothendieck groups, the assertions concerning the induced map on $K_0$ follow,  
including the relation between the ranks and degrees.
\end{proof}



\subsection{Connection with the derived category of $\CC_4$}\label{C4-conn-sec}

Let $X$ be a del Pezzo surface of degree $4$ with a geometric basis $(h,e_1,\ldots,e_5)$, and let
$(f,s,E_1,\ldots,E_4)$ be the corresponding $\P^1\times\P^1$-type basis (see \eqref{s-f-marking-eq}).

\begin{prop}\label{C4-orth-prop}
(i) There is an equivalence 
$$\Phi_f:D^b(\Coh \CC)\rTo{\sim}\lan \OO_X(-f),\OO_X\ran^{\perp} $$
for a weighted projective line $\CC=\CC_4$, such that
$$\Phi_f(\OO_\CC(-2))=\OO_X(-s-f), \ \ \Phi_f(\OO_\CC(-1))=\OO_X(-s), \ \ \Phi_f(S'_i)=\OO_{E_i}(-1), \ i=1,\ldots,4.$$

\noindent
(ii) One has
$$\rk_X(\Phi_f(V))=\rk_\CC(V), \ \ \deg_X(\Phi_f(V))=\deg_\CC(V),$$
and the map $\det\circ \Phi_f:K_0(\CC_4)\to\Pic(X)$ is given by
$$\det\circ \Phi_f:r\cdot ([\OO_\CC]-\frac{1}{2}\sum \eps_i)+d\cdot \frac{1}{2}([\OO_\CC]-[\OO_\CC(-1)])+\sum y_i\eps_i\mapsto r(f-\frac{Q}{2})+d\frac{f}{2}+\sum_i y_i(\frac{f}{2}-E_i).$$
The restriction of $\det\circ \Phi_f$ to $Q(D_4)\simeq \ker(\rk_\CC,\deg_\CC)\sub K_0(\CC_4)$
coincides with the map $A:Q(D_4)_{\Q}\to Q(D_5)_{\Q}\sub \Pic(X)$ given by
$$\eps_1\mapsto \frac{f}{2}-E_1=\frac{1}{2}(-\eps_1+\eps_2+\eps_3+\eps_4),$$
$$\eps_2\mapsto \frac{f}{2}-E_2=\frac{1}{2}(\eps_1-\eps_2+\eps_3+\eps_4),$$
$$\eps_3\mapsto \frac{f}{2}-E_3=\frac{1}{2}(\eps_1+\eps_2-\eps_3+\eps_4),$$
$$\eps_4\mapsto \frac{f}{2}-E_4=\frac{1}{2}(-\eps_1-\eps_2-\eps_3+\eps_4),$$
where on the right we use the basis $(\eps_i)_{i=1,\ldots,5}$ coming from the geometric marking on $X$ (see \eqref{eps-i-geom-eq}).
The map $A$ sends $B(r,d)$ to $B(d,r+d)\sub\bigoplus_{i=1}^4 \R\eps_i$.
\end{prop}

\begin{proof}
(i) Viewing $X$ as the blow up of $\P^1\times \P^1$ at $4$ points $q_1,\ldots,q_4$, gives the standard full exceptional collection in $D^b(\Coh X)$,
$$(\OO_{E_i}(-1)[-1])_{i=1,\ldots,4}, \OO_X(-s-f), \OO_X(-s), \OO_X(-f), \OO_X).$$
Deleting the last two objects we get a full exceptional collection in $\lan \OO_X(-f),\OO_X\ran^{\perp}$.

Let $F_i$ denote the exceptional curve with the class $f-E_i$ (the proper transform of the ruling $f$ through $q_i$)
The right mutation of $\OO_{E_i}(-1)[-1]$ through $\OO_X(-s-f)$ is  $\OO_X(-s-f+E_i)$, and the exact sequence
$$0\to \OO_X(-s-f+E_i)\to \OO_X(-s)\to \OO_X(-s)|_{F_i}\to 0$$
shows that the right mutation of $\OO_X(-s-f+E_i)$ through $\OO_X(-s)$ is $\OO_X(-s)|_{F_i}$.
Thus, we get a full exceptional collection
\begin{equation}\label{exc-coll-s-f-eq}
\OO_X(-s-f), \OO_X(-s), (\OO_X(-s)|_{F_i})_{i=1,\ldots,4}
\end{equation}
in $\lan \OO_X(-f),\OO_X\ran^{\perp}$. It is clear that this collection is strong and has the same algebra as the exceptional collection
\begin{equation}\label{C4-exc-coll-O(-2)-eq}
\OO_\CC(-2), \OO_\CC(-1), (S_i)_{i=1,\ldots,4}
\end{equation}
on the weighted projective line $\CC=\CC_4$ associated with the images of $q_1,\ldots,q_4$ under the projection to $\P^1$ given by $f$.
Hence, we get an equivalence $\Phi_f:D^b(\Coh \CC_4)\to \lan \OO_X(-f),\OO_X\ran^{\perp}$ sending the collection
\eqref{C4-exc-coll-O(-2)-eq} to \eqref{exc-coll-s-f-eq}.
Since, the left mutation of $S_i$ through $\lan \OO_\CC(-2),\OO_\CC(-1)\ran$ is $S'_i[-1]$, we have $\Phi_f(S'_i)=\OO_{E_i}(-1)$.

\noindent
(ii) The formulas for the rank and the determinant are checked on our exceptional collections. The assertions about the linear map $A$
are proved by a straightforward computation.
\end{proof}

\subsection{Existence of exceptional pairs with a given slope}

Let $X$ be a del Pezzo surface of degree $4$. As before, we fix a geometric marking $(h,e_1,\ldots,e_5)$, and denote by
$(f,s,E_1,\ldots,E_4)$ be the corresponding $\P^1\times\P^1$-type marking.

For coprime $a,b\in \Z$ and $y=\sum_{i=1}^4 y_i\eps_i\in B(a+b,a)$, we consider a class in $\Pic(X)_{\Q}$ given by
$$D_{ab}(y)=(a/2)f + (b/2)(Q-f)+y=(a/2)(Q/2-\eps_5)+(b/2)(Q/2+\eps_5)+y.$$
Note that if $[V]$ is an exceptional class in $K_0(\CC_4)$ then the equivalence $\Phi_f$ sends a class $[V]$ with rank $r$ and degree $d$
to a class in $K_0(X)$ with $c_1(\Phi_f[V])=D_{r+d,-r}(y)$ for some $y\in B(d,r+d)$ (see Proposition \ref{C4-orth-prop}).
In particular,  all classes $D_{ab}(y)$ are integral and form a single $W(D_4)$-orbit (for fixed $a,b$).

\begin{lemma}\label{W-orbit-lem}
All the elements
$D_{ab}(y)$, $D_{ba}(y')$ and $-D_{-a,-b}(y'')$ (for fixed $a,b$) are in the same $W(D_5)$-orbit.
\end{lemma}

\begin{proof}
The set $S$ of elements $w\in W(D_5)$ such that $w(\eps_5)=-\eps_5$ and $w$ fixes the subset $\{\pm \eps_i \ |\ i=1,\ldots,4\}$
forms a $W(D_4)$-coset (left or right), and for $w\in S$ (and $y\in B(a+b,a)$), we have
$$wD_{ab}(y)=(a/2)(Q/2+\eps_5)+(b/2)(Q/2-\eps_5)+w(y).$$
To see that this element is of the form $D_{ba}(y')$ with $y'\in B(a+b,b)$, it is enough to check that for some (and hence for all) $y\in B(a+b,a)$ and some $w\in S$ one has
$w(y)\in B(a+b,b)$. 
Since the action of $w\in S$ on $Q(D_4)_{\Q}$ is in the $W(D_4)$-coset (left or right) of the transformation $A_1$ sending $\eps_1$ to $-\eps_1$
and fixing $\eps_i$ for $i=2,3,4$, this follows from the fact that $A_1$ preserves $B(0,1)$ and swaps $B(1,0)$ with $B(1,1)$.

On the other hand, we have
$$-D_{-a,-b}(y)=(a/2)f+(b/2)(Q-f)-y=D_{a,b}(-y)$$
(note that the longest element of $W(D_4)$ takes $y$ to $-y$).
\end{proof}
 

The following is a more precise version of Theorem A (since $\deg(D_{d-r,r}(y))=\deg(D_{r,d-r}(y))=d$).

\begin{theorem}\label{main-thm}
For all relatively prime $(d,r)$, with $r>0$, and every $y\in B(d,d-r)$ (resp., $y'\in B(d,r)$)
there is a unique (up to isomorphism) exceptional bundle $V$ on $X$ of rank $r$ with $c_1(V)=D_{d-r,r}(y)$ (resp., $c_1(V)=D_{r,d-r}(y')$),
and that bundle satisfies $R\Hom(V,\OO_X)=0$.  Similarly, there is a unique exceptional bundle $V$ of rank $r$ with $c_1(V)=D_{d+r,-r}(y)$ (resp., $c_1(V)=D_{-r,d+r}(y')$),
and that bundle satisfies $R\Hom(\OO_X,V)=0$.
\end{theorem}


\begin{proof}
Uniqueness follows from the fact that an exceptional bundle $E$ is determined by its class in $K_0$, hence, if in addition $E\in \lan\OO_X\ran^\vee$, it is determined
by its rank and $c_1$.
By Lemma \ref{W-orbit-lem} and duality, 
it is enough to prove existence of an exceptional pair $(V,\OO_X)$ with $V$ of rank $r$ and $c_1(V)=D_{d+r,-r}(y)$. 
But this immediately follows from Proposition \ref{C4-orth-prop}.
\end{proof}

It is natural to ask whether our divisor classes $D_{a,b}(y)$ descend to del Pezzo surfaces of bigger degree (with respect to blow up maps).
 
\begin{prop}\label{Dab-descent-lem}
The divisor class of the form $D_{a,b}(y)$ is the pull-back of a class on a del Pezzo surface of degree $5$ if and only if either $|a|\le 2$ or $|b|\le 2$.
If  $a=\pm 1$ (resp., $b=\pm 1$), a class of the form $D_{a,b}(y)$ descends to $\P^1\times \P^1$ if $b$ is odd (resp. $a$ is odd); otherwise, it descends to the
Hirzubruch surface $F_1$.
The classes corresponding to the pairs $(a,b)=(1,2),(-1,-2),(2,1),(-2,-1)$ are the pull-backs from $\P^2$.
\end{prop}

\begin{proof}
A class descends to a del Pezzo surface of degree $5$ if and only if is orthogonal to the class of an exceptional curve. In terms of the basis $(Q,\eps_1,\ldots,\eps_5)$ of
$\Pic(X)_{\Q}$, the set of classes 16 exceptional curves is the $W(D_5)$-orbit of 
$$2h-\sum_{i=1}^5 e_i=Q/4+\sum_{i=1}^5 \eps_i/2.$$
In other words, all exceptional classes have form $e=Q/4+\sum_{i=1}^5 x_i\eps_i$, where $x_i=\pm 1/2$, with even number of $x_i$'s negative. 
Recall that the basis $(\eps_i)$ is orthonormal with respect to the negative of the intersection form. Hence,
the intersection number of the class $e$ with $Q/2-\eps_5$ (resp., $Q/2+\eps_5$) is $1/2+x_5$ (resp., $1/2-x_5$). Thus,
we have
$$e\cdot D_{a,b}(y)=\begin{cases}  a/2+x\cdot y, & x_5=1/2, \\  b/2+ x\cdot y, & x_5=-1/2.\end{cases}$$
Since $|x\cdot y|\le 1$, this can be zero only if either $|a|\le 2$ or $|b|\le 2$. 

By Lemma \ref{W-orbit-lem}, it remains to consider the cases $b=2$ and $b=1$, with some choice of $y$. In the case $b=2$, $a$ is odd, so $y\in B(1,1)$, i.e., $y_i=\pm 1/2$, 
with odd number of negative $y_i$'s. Hence, $e=Q/4+y-\eps_5/2$ is an exceptional class with $e\cdot D_{a,b}(y)=0$ (recall that $y\cdot y=-1$). 

In the case where $b=1$ and $a$ is odd, we can take $y=-\eps_4$. Then
$$D_{a,1}(\eps_4)=\frac{a}{2}(Q/2-\eps_5)+\frac{1}{2}(Q/2+\eps_5)-\eps_4=\frac{a-1}{2}(Q/2-\eps_5)+Q/2-\eps_4=\frac{a-1}{2}f+s$$
which comes from $\P^1\times \P^1$.

In the case where $b=1$ and $a$ is even, we can take $y=-\sum_{i=1}^4 \eps_i/2$.
Then
$$D_{a,1}(-\sum_{i=1}^4\eps_i/2)=\frac{a}{2}(3Q/4-\sum_{i=1}^4\eps_i/2-\eps_5/2)-(\frac{a}{2}-1)(Q/4-\sum_{i=1}^4\eps_i/2+\eps_5/2)=\frac{a}{2}h-(\frac{a}{2}-1)e_5,$$
which comes from $F_1$, and even from $\P^2$ in the case $a=2$.
\end{proof}


Combining Proposition \ref{Dab-descent-lem} with Lemma \ref{dP-lem}(ii) we derive the following corollary about descending our exceptional bundles to del Pezzo surfaces
of bigger degree.

\begin{cor}\label{exc-bun-descent-rem}
The exceptional pairs $(\OO_X,V)$ constructed in Theorem \ref{main-thm} 
descend to a del Pezzo surface of degree $5$ provided $r=2$ or $r=d\pm 2$. Similarly, in the case $r=1$ or $r=d\pm 1$ they descend to a del Pezzo surface of degree $8$
(to $\P^1\times\P^1$ if $d$ is even, and to $F_1$ otherwise).
\end{cor}

\begin{remark}
The statement of Theorem \ref{main-thm} also holds for $r=0$, with the sole exception that the sheaf one obtains is not a bundle but is supported on an exceptional
curve. Namely, the classes $D_{1,0}(y)$, for $y\in B(1,1)$, are exactly the $8$ exceptional classes that have trivial intersection with $f$.
\end{remark}

\subsection{Non-sporadic range}

We refer to exceptional pairs $(\OO_X,V)$ or $(V,\OO_X)$, that do not have $c_1(V)$ described in Theorem \ref{main-thm} for some geometric marking, 
as {\it sporadic}. The next result gives a range of values $(r,d)$ for which sporadic pairs do not appear.
As before, $X$ denotes a del Pezzo surface of degree $4$.

\begin{theorem}\label{sporadic-thm}
Let $V$ be an exceptional bundle on $X$ of rank $r$ and slope $\mu=d/r$, such that the pair $(V,\OO_X)$ is exceptional. Assume that 
$$\mu\not\in[-r-\frac{1}{r}-4, r+\frac{1}{r}].$$
Then 
\begin{itemize}
\item either $c_1(V)=D_{d+r,-r}(y)$ with respect to an appropriate geometric marking,
\item or $V\simeq p^*\Om_{\P^2}(-1)$, where $p:X\to \P^2$ is an appropriate blowdown (in this case $r=2$, $\mu=-15/2$),
\item or $V\simeq R(p^*\Om_{\P^2}(-1))\simeq (p^*T_{\P^2}(1))(-Q)$ (in this case $r=2$, $\mu=7/2$).
\end{itemize} 
\end{theorem}

\begin{proof}
Since $\mu(R(V))=-\mu-4$, it is enough to prove the assertion assuming that $\mu>r+1/r$.
Since an exceptional bundle in $\lan \OO_X\ran^\perp$ is determined by its rank and $c_1$, it is enough to check that for this range
either $c_1(V)=D_{d+r,-r}(y)$ or $(r,d)=(2,7)$ and $c_1(V)=5h-2Q$ for some geometric marking.

Recall that by Proposition \ref{C5-orth-prop}, we have an equivalence $\Phi:D^b(\Coh \CC)\rTo{\sim} \lan\OO_X\ran^\perp$, where $\CC=\CC_5$ is a weighted projective line.
Let $V'=\Phi^{-1}(V)\in D^b(\Coh\CC)$. Then $r':=\rk_\CC(V')=2r+d>0$ and $d':=\deg_\CC(V')=-r$, so changing
$\Phi$ by an even shift we can assume $V'$ to be a vector bundle on $\CC$. Our assumption $d/r>r+1/r$ is equivalent to $r'>(1-d')^2$.

Applying Proposition \ref{C5-omnipresent-prop} we deduce that either $V'$ is right orthogonal to one of the objects $(S_i,S'_i)$, or
$(r,d)=(2,7)$ and $\pi[V']$ is in the $W(D_5)$-orbit of $-\frac{5}{2}\sum_{i=1}^5 \eps_i$. 
In the latter case, we get that with respect an appropriate geometric marking,
$$c_1(V)=\frac{7}{4}Q-\frac{5}{2}\sum_{i=1}^5 \eps_i=5h-2Q$$
(the coefficient of $Q$ is determined by the degree of $V$).

In the former case we get that either $\Hom^*(S_i,V')=0$ or $\Hom^*(S'_i,V')=0$ for some $i$.
Applying the equivalence $\Phi$, we get that either $\Hom^*(\OO_X(e_i-h),V)=0$ or $\Hom^*(\OO_X(Q-e_i+h),V)=0$ for some $i$.
But the classes $e_i-h=-Q/2+\eps_i$ and $-Q-e_i+h=-Q/2-\eps_i$, where $i=1,\ldots,5$, form a single $W(D_5)$-orbit.
Thus, choosing an appropriate geometric marking, we can assume that $\Hom^*(\OO_X(-f),V)=0$ (recall that $f=h-e_5$).
But then $V$ is in the subcategory $\lan \OO_X(-f),\OO_X\ran^\perp$, which by Proposition \ref{C4-orth-prop} is equivalent to $D^b(\Coh \CC_4)$.
Finally, the description of exceptional classes in $K_0(\CC_4)$ (see Lemma \ref{C4-exc-lem}) implies that $c_1(V)$ is of the form $D_{d+r,-r}(y)$. 
\end{proof}


\section{Exceptional pairs on del Pezzo surfaces of degree $\ge 5$}

For each $k$, $5\le k\le 9$, it is easy to check that not every slope $\mu=d/r$ occurs as $\mu(V)$ for an exceptional pair $(V,\OO)$.

\begin{prop}\label{qu-ineq-prop} 
Let $V$ be an exceptional bundle on a del Pezzo surface of degree $k$, where $k\ge 5$, such that $(V,\OO)$ is an exceptional
pair. Then $d=\deg(V)$ and $r=\rk(V)$ satisfy
\begin{equation}\label{spor-slope-ineq}
d^2+krd+kr^2\ge -k.
\end{equation}
In particular, for $k$ odd, the case $d=(-kr+1)/2$ does not occur for any odd rank $r\ge 3$; while for $k$ even, the case $d=-\frac{k}{2}r+1$ does not
occur for any $r\ge 2$.
\end{prop}

\begin{proof}
For an exceptional bundle $E$ such that $\chi(E)=0$, we get from \eqref{chi-V-eq} that
$$1=\chi(E,E)=-r^2-c_1^2-rd,$$
where $r=\rk(E)$, $d=c_1(E)\cdot Q$.
We have 
$$c_1(E)=\frac{d}{k}Q+\a,$$ where $\a\in \lan Q\ran^\perp\sub \Pic(X)_{\Q}$. Then we can rewrite the above identity as
$$\frac{d^2}{k}+rd+r^2=-1-\a^2.$$
Since $\a^2\le 0$ by the Hodge index theorem, this gives the claimed inequality.
\end{proof}

Note that in the case $k=9$ (i.e., for $\P^2$) the inequality \eqref{spor-slope-ineq} becomes an equality
\begin{equation}\label{P2-constraint-eq}
d^2+9rd+9r^2=-9,
\end{equation}
where $d=\deg(V)$ is divisible by $3$.
In this case, using the fact that all exceptional objects in $\lan \OO_{\P^2}\ran^\perp$ lie in the helix generated by $(\OO_{\P^2}(-2),\OO_{\P^2}(-1))$,
one can determine all the slopes $\mu(V)$ that occur for $(V,\OO)$. 

More generally, for $k\ge 5$, we can consider relatively prime
$(d,r)$ satisfying
\begin{equation}\label{forb-int-eq}
-k\le d^2+krd+kr^2\le -1.
\end{equation}
We will prove that all of them arise as $(\deg(V),\rk(V))$ from exceptional pairs $(V,\OO)$ (with the restriction that $d$ is divisible by $3$ if $k=9$).

\begin{prop}\label{higher-deg-dP-prop}
(i) Let $X_k$ be a del Pezzo surface of degree $k$, where $5\le k\le 7$. Then for every relatively prime $(d,r)$, with $r>0$, satisfying
\eqref{forb-int-eq}, there exists an exceptional pair $(V,\OO)$ on $X_k$ with $r=\rk(V)$, $d=\deg(V)$.

\noindent
(ii) In the case $k=8$, for $X_8\simeq F_1$ (the Hirzebruch surface) any relatively prime pair $(d,r)$ satisfying \eqref{forb-int-eq}, 
arises from an exceptional pair $(V,\OO)$ on $X_8$. In the case $X_8\simeq \P^1\times \P^1$, such a pair $(d,r)$ arises from an exceptional pair on $X_8$
if and only if $d$ is even.

\noindent
(iii) For $X=\P^2$ we can realize in this way every pair $(d,r)$ satisfying \eqref{P2-constraint-eq}.
\end{prop}

\begin{proof}
Let us consider the quadratic order of discriminant $k(k-4)$,
$$O_k:=\Z+\Z\frac{-k+\sqrt{k(k-4})}{2}\sub \Q(\sqrt{k(k-4)}),$$
and let $\si:O_k\to O_k$ denote the Galois conjugation.
For every $V$ with $(\deg(V),\rk(V))=(d,r)$, we set
$$\xi(V)=\xi(d,r):=-d+r\frac{-k+\sqrt{k(k-4)}}{2}\in O_k.$$
Note that
$$\Nm(\xi(V))=d^2+krd+kr^2.$$
 
Let us consider two operations \eqref{mu-transf-eq} on exceptional pairs $(V,\OO)$ (where we allow $V$ to be any exceptional object of the derived category).
It is easy to check that
\begin{equation}\label{conj-unit-op-eq}
\xi(R(V))=-\si\xi(V), \ \ \xi(RM(V))=u\cdot \xi(V),
\end{equation}
where $u=-1+k/2+\sqrt{k(k-4)}/2$ is a unit in $O_k$.

Now we consider the case of each $k$ separately.

\noindent
{\bf Case $k=9$.}
In this case $O_9=\Z+\Z\cdot 3\frac{-1+\sqrt{5}}{2}$ has conductor $3$ in the ring of integers $O\sub \Q(\sqrt{5})$. 
Note that $u=u_f^4$, where $u_f=(1+\sqrt{5})/2$ is the fundamental unit in $O$.
We are looking for elements of $O_9$ of norm $-9$. Since $3$ does not split in $O$, they are of the form $\pm 3u_f^{2n+1}$.
Thus, up to the action of the operations \eqref{conj-unit-op-eq}, we can reduce to the case $\xi(V)=\pm 3u_f$, in which case
$V$ is a shift of $\OO(-2)$.

\noindent
{\bf Case $k=8$.}
In this case $O_8=\Z+\Z\cdot 2\sqrt{2}$ has conductor $2$ in $O=\Z[\sqrt{2}]$, and
$u=u_f^2$, where $u_f=1+\sqrt{2}$ is the fundamental unit in $O$. We are looking for elements of $O_8$ with norm in $[-8,-1]$.
Since a norm of an element in $O_8$ is a square modulo $8$, the norm is actually either $-8$, $-7$, or $-4$.

If $\Nm(a)=-8=-(\sqrt{2})^6$ for $a\in O$, then $a=(\sqrt{2})^3v$, where $v\in O$ has $\Nm(v)=1$, so $v=\pm u_f^{2n}$. Hence, our operations \eqref{conj-unit-op-eq}
(together with the shift) reduce to the case $\xi(V)=2\sqrt{2}$, i.e., $\rk(V)=1$, $\deg(V)=-4$.
Similarly, if $\Nm(a)=-7$, then $a=(-1+2\sqrt{2})v$, where $\Nm(v)=1$, so we can reduce to $\xi(V)=-1+2\sqrt{2}$, i.e., $\rk(V)=1$, $\deg(V)=-3$.
Finally, if $\Nm(a)=-4$ then $a=2v$, where $v=\pm u_f^{2n+1}$, so we can reduce to $\xi(V)=2(-1+\sqrt{2})$, i.e., $\rk(V)=1$, $\deg(V)=-2$.

Thus, the question reduces to which of these degrees can occur for a line bundle $L$ forming an exceptional pair $(L,\OO)$. For $X_8=\P^1\times\P^1$, the degree
is always even, and both degrees $-2$ and $-4$ occur (for $L=\OO(-1,0)$ and $L=\OO(-1,-1)$). On the other hand, it is easy to see that for $X_8=F_1$, the exceptional
pair $(L,\OO)$ should have $L=\OO(-s-nf)$ or $L=\OO(-f)$, where $f$ is the class of a fiber of $F_1\to \P^1$, and $s$ is the class of a section, with $s^2=-1$. Since $Q=2s+3f$,
in the former case we have $\deg(L)=2n-1$, while in the latter case $\deg(L)=-2$. Thus, for $X_8=F_1$, all three degrees occur.

\noindent
{\bf Case $k=7$.}
In this case $O_8=O=\Z+\Z\frac{\sqrt{21}-1}{2}$ is the ring of integers in $\Q(\sqrt{21})$, and we are looking for elements of $O$ with the norm in $[-7,-1]$.
We have $u=u_f=\frac{5+\sqrt{21}}{2}$.

Since the norm is a square modulo $7$, the norm has to be in $\{-7,-6,-5,-3\}$. It is easy to rule out $-7$. Also, $d^2+7rd+7r^2$ can be even only if both $d$ and $r$ are even.
Solutions of $\Nm(a)=-5$ are of the form $a=\pm \frac{-1+\sqrt{21}}{2}\cdot u_f^n$, so we can reduce the problem to $\xi(V)=\frac{-1+\sqrt{21}}{2}$, i.e., $\rk(V)=1$, $\deg(V)=-3$.
Similarly, solutions of $\Nm(a)=-3$ are of the form $a=\pm \frac{-3+\sqrt{21}}{2}\cdot u_f^n$, so we can reduce to $\xi(V)=\frac{-3+\sqrt{21}}{2}$, i.e., $\rk(V)=1$, $\deg(V)=-2$.

In both cases we can take $(V,\OO)$ to be the pull-back of the similar pair on $F_1$ under the blow up map $X_7\to X_8\simeq F_1$.

\noindent
{\bf Case $k=6$.}
In this case $O_6=\Z[\sqrt{3}]$, we have $u=u_f=2+\sqrt{3}$, and we are looking for elements with the norm in $[-6,-1]$.
It is easy to see that the only two possibilities are $\Nm(a)=-3$ and $\Nm(a)=-2$. The solutions are of the form $\pm \sqrt{3}u_f^n$ and $\pm (-1+\sqrt{3})u_f^n$,
so we reduce to the cases $(r,d)=(1,-3)$ and $(r,d)=(1,-2)$. Both cases are realized as pull-backs under the blow up map $X_6\to X_7$.

\noindent
{\bf Case $k=5$.}
In this case $O_5=O=\Z+\Z\frac{-1+\sqrt{5}}{2}$, and $u=u_f^2$, where $u_f=(1+\sqrt{5})/2$. The possible norms in $[-5,-1]$ (with $(r,d)$ relatively prime)
are $-5$ and $-1$. The norm $-5$ is realized by elements of the form $\pm \sqrt{5}u_f^{2n}$, while the norm $-1$ is realized by elements of the form $\pm u_f^{2n+1}$.
Thus, we reduce to the cases $(r,d)=(2,-5)$ and $(r,d)=(1,-2)$. The latter case is realized as the pull-back under the blow up map $X_5\to X_6$.

It remains to find an exceptional pair $(V,\OO)$ with $\rk(V)=2$ and $\deg(V)=-5$. For this we can realize $X=X_5$ as the linear section of $G(2,5)$ and take the restriction to $X$
of the exceptional pair $(\UU,\OO)$ on $G(2,5)$, where $\UU\sub \OO^5$ is the universal subbundle (the needed cohomology on $X$ are easily computed using the Koszul
resolution for $\OO_X$ on $G(2,5)$).
\end{proof}

\begin{remark} 
For most values $(d,r)$, the pull-backs of the exceptional pairs $(V,\OO_{X_k})$ considered in Proposition \ref{higher-deg-dP-prop} to a del Pezzo surface of degree $4$,
are sporadic (i.e., do not appear from Theorem \ref{main-thm}). Indeed, by Proposition \ref{Dab-descent-lem}, this happens as soon as $r>2$ and $d<-r-2$.
\end{remark}

\section{Application to bihamiltonian structures}

\subsection{Nodal anticanonical divisors on del Pezzo surfaces}\label{pencils-sec}


\begin{lemma}\label{sing-pt-lem}
Let $X$ be a del Pezzo surface of degree $k\ge 3$.
Then there exists a nonempty open subset $U\sub X$, such that for every $q\in U$, there exists an irreducible anticanonical curve $C\sub X$, such that $q$
is a singular point of $C$.
\end{lemma}

\begin{proof} 
In the case $k=3$, $X$ is a cubic surface in its anticanonical embedding into $\P^3$. Then we can take $U\sub X$ to be the complement to the union of $27$ lines on $X$.
For every $q\in U$, let $H_q\sub \P^3$ be the tangent plane to $X$.
The intersection $C_q:=H_q\cap X$ is an anticanonical divisor, singular at $q$. Since $C_q$ is a cubic in $H_q$, and $q$ does not lie on a line contained in $X$, $C_q$ is irreducible.

In the case $k>3$, let $\pi:X'\to X$ be the blow up map, where $X'$ is a smooth cubic surface. Then we can take $U\sub X$ to be the complement to the image of the union of the lines on $X'$.
\end{proof}

The following lemma describes a special configuration of $6$ points in $\P^2$ that appears in characteristic $2$.
Let us denote by $X_3^F\sub\P^3$ the Fermat cubic $x_0^3+x_1^3+x_2^3+x_3^3=0$.

\begin{lemma}\label{X3F-lem}
Assume the characteristic is $2$.
Let $p_1,\ldots,p_6$ be an unordered configuration of points in $\P^2$, such that no $3$ are collinear. 
For $i=1,\ldots,6$, let $C_i$ denote the unique conic passing through $S_i:=\{p_1,\ldots,p_6\}\setminus\{p_i\}$, 
Assume that for every $i$, $C_i$ is tangent to every line through $p_i$. Then $p_1,\ldots,p_6$ is projectively equivalent
to the unique (up to $\PGL_3({\mathbb F}_4)$) unordered configuration of $6$ points on $\P^2({\mathbb F}_4)$,
such that no $3$ are collinear. The blow up of $\P^2$ at these $6$ points is isomorphic to $X_3^F$.
\end{lemma}

\begin{proof}
Without loss of generality we may assume that the first four points are
\begin{equation}\label{4pts-coord-eq}
p_1=(1:0:0), \ \ p_2=(0:1:0), \ \ p_3=(0:0:1), \ \ p_4=(1:1:1).
\end{equation}
Let $p_5=(x_0:y_0:1)$, $p_6=(x_1:y_1:1)$. Then the equation of $C_6$ is $q_6(x,y,z)=az(x+y)+y(x+z)=0$, where
$a=y_0(x_0+1)/(x_0+y_0)$. Hence, $p_6$ is the point where all derivatives of $q_6$ vanish,
which gives $x_1=a+1$, $y_1=a$. In particular, $x_1+y_1+1=0$. Exchanging the roles of $p_5$ and $p_6$ we see that $x_0+y_0+1=0$.
Hence, $a=x_0^2+1$, so 
$$p_5=(x_0:x_0+1:1), \ \ p_6=(x_0^2:x_0^2+1:1).$$
Exchanging the roles of $p_5$ and $p_6$ we deduce that $x_0^4=x_0$, hence all points are defined over ${\mathbb F}_4$.
Since no three points are collinear, we see that $x_0\neq 0,1$. Hence, $x_0$ is one of two roots of $x_0^2+x_0+1=0$. 
Hence, coordinates of all points are in ${\mathbb F}_4$.

Conversely, suppose we have $6$ points on $\P^2({\mathbb F}_4)$, such that no three are collinear, with the first $4$ points given by \eqref{4pts-coord-eq},
and $p_5=(x_0:y_0:1)$, $p_6=(x_1:y_1:1)$. Then $x_i,y_i\in {\mathbb F}_4\setminus \{0,1\}$ and $x_i\neq y_i$ for $i=1,2$, $x_1\neq x_0$. Hence, the points $p_5$ and
$p_6$ are uniquely determined up to permutation. 

On the other hand, all $27$ lines on the Fermat cubic $X_3^F$ are defined over ${\mathbb F}_4$. Thus, for any choice of the blow up morphism
$\pi:X_3^F\to \P^2$ the corresponding $6$ points in $\P^2$ are defined over ${\mathbb F}_4$. Hence, $X_3^F$ is isomorphic to the blow up of $\P^2$ at the above configuration
of $6$ points in $\P^2({\mathbb F}_4)$.
\end{proof}

Let $X$ be a del Pezzo surface of degree $k\ge 3$, $P:=|Q|\simeq \P^k$ the anticanonical linear system.
Let $\wt{P}\sub P\times X$ denote the incidence variety of $(C,p)$ such that $p\in C$.
We consider the following loci in $P$ and $\wt{P}$:
\begin{itemize}
\item
$P_{sing}$ (resp., $\wt{P}_{sing}$) is the locus of singular divisors (resp., of pairs $(C,p)$ such that $p$ is a singular point of $C$);
\item
$P_{red}\sub P_{sing}$ is the locus of reducible divisors;
\item
$P_{cusp}\sub P_{sing}$ is the locus of non-nodal divisors.
\end{itemize}

\begin{lemma}\label{sing-locus-lem} Assume that $k\ge 3$.

\noindent 
(i) The loci $P_{sing}$ (resp., $\wt{P}_{sing}$), $P_{red}$ and $P_{cusp}$ are closed. The varieties
$\wt{P}_{sing}$ and $P_{sing}$ are irreducible of dimension $k-1$. 

\noindent
(ii) One has $P_{red}\neq P_{sing}$. Assume that $X\not\cong X^F_3/\bar{{\mathbb F}}_2$.
Then $P_{cusp}\neq P_{sing}$.
\end{lemma}

\begin{proof}
(i) Since $-K_X$ is very ample, for a fixed point $p\in X$, the condition that an anticanonical divisor $C$ is singular at $p$ is given by three independent linear conditions
on a point in $P$. Hence, $\wt{P}_{sing}$ is closed in $P\times X$ and is a projective bundle over $X$,
with fibers $\P^{k-3}$, so $\wt{P}_{sing}$ is irreducible of dimension $k-1$. The subset $P_{sing}\sub P$ is the image of $\wt{P}_{sing}$, so it is closed.

Reducible anticanonical divisors of the form $C_1+C_2$, with fixed rational equivalence classes of $C_1$ and $C_2$, are in the image of the map
$|C_1|\times |C_2|\to P$, so they form a closed subset. We claim that there is finitely many possibilities for rational equivalence classes of $C_i$. 
Indeed, we can assume that neither $C_1$ or $C_2$ is a $(-1)$-curve. Then representing $X$ as a blow up of a set of points $S\sub \P^2$,
we get that $C_1+C_2$ is the proper transform of a reducible cubic passing through $\P^2$, so the class of one of the components is either $h-e_i$ or $h-e_i-e_j$.
This implies the claim and proves that $P_{red}$ is closed.
It is well known that the locus of curves in $P$ with at most nodal singularities is open, hence, $P_{cusp}$ is closed.

\noindent
(ii) Lemma \ref{sing-pt-lem} shows that $P_{red}\neq P_{sing}$.
It remains to prove existence of a nodal anticanonical divisor $C\sub X$ under our assumptions.
For $k\ge 4$, $X$ can be realized as the blow up of $\P^2$ in a set $S$ of $\le 5$ points (in general linear position), so we can take as $C$ the proper transform of 
$\ov{C}=L_1\cup L_2\cup L_3$, the nodal union of three lines in $\P^2$, so that $S\sub \ov{C}$ but none of the nodes of $\ov{C}$ is in $S$. For example, if
$S=\{p_1,\ldots,p_5\}$, we can take $L_1$ to be the line through $p_1,p_2$, $L_2$ the line through $p_3,p_4$, and $L_3$ a generic line through $p_5$.


Now assume $k=3$ and characteristic is $\neq 2$. Let $X$ be the blow up of $\P^2$ at points $p_1,\ldots,p_6$.
Consider the (smooth) conic $C\sub \P^2$ through $p_1,\ldots,p_5$, and let $\ell_i\sub \P^2$, for $i=1,\ldots,5$, be the line through $p_i$ and $p_6$.
Note that the linear projection of $C$ from $p_6$ is a degree $2$ map $C\to \P^1$, so it has at most two ramification points, which correspond to tangent lines to $C$
passing through $p_6$. Since the five lines $\ell_1,\ldots,\ell_5$ all pass through $p_6$ and are distinct, one of them is not tangent to $C$ 
Say, $\ell_1$ is not tangent to $C$. Therefore, $C\cap \ell_1$ consists of two points, $p_1$ and $q$, where $q$ is distinct from $p_i$.
Consider the proper transforms $\wt{C}$ and $\wt{\ell}_1$ of $C$ and $\ell_1$ in $X$. Then $e_1\cup \wt{C}\cup \wt{\ell}_1$ is the nodal anticanonical curve.

In the case of $k=3$ and characteristic $2$, the only case when the above argument does not go through is when the linear projection of $C$ from each point $p_6$ is purely 
inseparable, i.e., every line through $p_6$ is tangent to $C$, and similarly for other points $p_i$ instead of $p_6$. By Lemma \ref{X3F-lem}, this implies
that $X\simeq X_F^3$.
\end{proof}


\begin{prop}\label{integral-curve-prop}
Let $X$ be a del Pezzo surface of degree $k\ge 3$, such that $X\not\cong X^F_3/\bar{{\mathbb F}}_2$.
Then there exists a nonempty open subset $U\sub X$, such that for every $q\in U$
there exists an integral nodal anticanonical curve $C\sub X$ such that $q$ is the node of $C$.
\end{prop}

\begin{proof}
By Lemma \ref{sing-locus-lem}, integral nodal curves are dense in $P_{sing}$. 
By Lemma \ref{sing-pt-lem}, the projection $\wt{P}_{sing}\to X$ is dominant.
Hence, the restriction to the non-empty open subset of $(C,p)\in \wt{P}_{sing}$ with $C$ integral nodal is also dominant.
\end{proof}

\begin{lemma}\label{blow-up-lem}
Let $X$ be a weak del Pezzo surface with $K_X^2>1$, $C\sub X$ an irreducible anticanonical divisor, $p\in C$ a smooth point. Then the blow up $\wt{X}$ of $X$ at $p$
is still a weak del Pezzo surface, the proper transform $\wt{C}\sub \wt{X}$ of $C$ is irreducible, and the projection $\wt{C}\to C$ is an isomorphism.
\end{lemma}

\begin{proof}
By \cite[Prop.\ 8.1.23]{Dolg}, it is enough to check that $p$ does not lie on any $(-2)$-curve $C'$. But this immediately follows from $C\cdot C'=0$.
\end{proof}

\begin{cor}\label{non-isotrivial-cor}
Under the assumptions of Proposition \ref{integral-curve-prop},
every anticanonical curve $C_0\sub X$ is contained in a non-isotrivial pencil of anticanonical curves.
\end{cor}

\begin{proof}
By Proposition \ref{integral-curve-prop}, there exists an integral nodal anticanonical curve $C\sub X$. In addition, 
we can assume that the node of $C$ is not contained in $C_0$. We claim that then the pencil $\lan C_0,C\ran$ is non-isotrivial.

Indeed, applying Lemma \ref{blow-up-lem} to blowing up points $C\cap C_0$, and replacing $C_0$ and $C$ by their proper transforms,
we reduce to the case when $X$ is a weak del Pezzo surface of degree $1$. Then blowing up the unique point in $C_0\cap C$
gives a minimal elliptic surface,
such that the proper transform of $C$ is still integral nodal. By Tate's Algorithm (see e.g. \cite[Sec.\ IV.9]{Silverman}), this implies non-isotriviality.
\end{proof}

\subsection{Compatible Poisson brackets}


\begin{proof}[Proof of Theorem B]
Start with an elliptic curve $C$. It is easy to see that for every $k\ge 4$, $C$ can be realized as an anticanonical divisor on a del Pezzo surface $X_k$ of degree $k$
(where in the case $k=8$ we can take $X_8=F_1$). Namely, start with an embedding $C\sub \P^2$ and $5$ generic points $p_1,\ldots,p_5\in C$, so that no three are collinear.
Then the blow up $X_4$ of $\P^2$ at these $5$ points is a del Pezzo surface of degree $4$, and $C$ lifts to an anticanonical curve of $X_4$. By blowing up a subset of these $5$ 
points, we can embed $C$ into $X_k$ with $k\ge 4$.

Let $(\OO,V)$ be an exceptional pair of vector bundles on $X_k$, where $\rk(V)=r$ and $\deg(V)=d>r+1$.
Applying \cite[Thm.\ 4.4(i)]{HP-bih}, we get a linear map 
$$\kappa:H^0(X_k,Q)\to H^0(\P^{d-1},{\bigwedge}^2 T),$$ 
so that every element corresponding to a smooth anticanonical divisor $D$ maps to the corresponding Poisson bracket of type $q_{d,r}(D)$, and all the brackets
in the image of $\kappa$ are compatible.
Furthermore, by \cite[Thm.\ 4.4(ii)]{HP-bih}, $\kappa$ is injective provided every singular anticanonical divisor on $X_k$ extends to a non-isotrivial anticanonical pencil.
By Corollary \ref{non-isotrivial-cor} the latter condition is always satisfied since $k\ge 4$.

Now part (i) follows from Theorem A. Namely, we can find $(\OO,V)$ on $X_4$ with given $(d,r)$. Then $\kappa(C)$ is exactly $q_{d,r}(C)$, and the image of $\kappa$
is the required linear subspace of dimension $\dim H^0(X_4,Q)=5$. In the case $r=2$ or $r=d-2$, we first descend $(\OO,V)$ to a degree $5$ del Pezzo $X_5$, see
Remark \ref{exc-bun-descent-rem}.
 
Similarly, part (ii) follows from Proposition \ref{higher-deg-dP-prop} (applied to $(-d,r)$ since this Proposition is formulated in terms of the dual pair $(V^\vee,\OO)$).
Note that in the case $k=8$ we take $X_8=F_1$, so that indeed any $(d,r)$ in the given range can be realized on $X_8$.

For part (iii), we use the well known linear map 
$$s:H^0(\P^{d-1},{\bigwedge}^2 T)\to H^0(\A^d, {\bigwedge}^2 T)^{\G_m}$$
sending Poisson brackets on $\P^{d-1}$ to quadratic Poisson brackets on $\A^d$ (see \cite{Bondal}, \cite[Sec.\ 12]{P-Pois}).
This gives a lift of a linear subspace of compatible Poisson brackets on $\P^{d-1}$ to a similar subspace on $\A^d$.
Recall also that by \cite[Thm.\ 6.12]{HP-bos}, a nonzero FO bracket $\pi$ on $\P^{d-1}$ admits no nonzero Poisson vector fields.
Hence by \cite[Thm.\ 12.1]{P-Pois}, a nonzero FO bracket admits a unique lifting to a quadratic Poisson bracket on $\A^d$, which proves the uniqueness statement.
\end{proof}



\begin{thebibliography}{9}
\bibitem{Bondal} A.~I.~Bondal, {\it Non-commutative deformations and Poisson brackets on projective spaces}, MPI preprint 93-67.
\bibitem{Bour} N.~Bourbaki, {\it Groupes et alg\`ebres de Lie}, ch.\ IV--VI, Hermann, Paris, 1968.
\bibitem{Dolg} I.~V.~Dolgachev, {\it Classical Algebraic Geometry: a modern view},
Cambridge University Press, Cambridge, 2012.
\bibitem{FO} B.~L.~Feigin, A.~V.~Odesskii, {\it Vector bundles on an elliptic curve and Sklyanin algebras}, in 
{\it Topics in quantum groups and finite-type invariants}, 65--84, Amer. Math. Soc., Providence, RI, 1998.
\bibitem{Gor} A.~L.~Gorodentsev, {\it Exceptional bundles on surfaces with a moving anticanonical class}, Math. USSR-Izv. 33 (1989), no. 1, 67--83.
\bibitem{HP-bih} Z.~Hua, A.~Polishchuk,  {\it Elliptic bihamiltonian structures from relative shifted Poisson structures}, J. Topology 16 (2023), 1389--1422.
\bibitem{HP-bos} Z.~Hua, A.~Polishchuk, {\it Bosonization of Feigin-Odesskii Poisson varieties}, arXiv:2306.14719.
\bibitem{KO} S.~A.~Kuleshov, D.~O.~Orlov,
{\it Exceptional sheaves on Del Pezzo surfaces}, 
Russian Acad. Sci. Izv. Math. 44 (1995), no.3, 479--513.
\bibitem{Lenzing} H.~Lenzing, {\it Weighted projective lines and applications}, EMS Ser. Congr. Rep., EMS, Z\"urich, 2011, 153--187.
\bibitem{LM} H.~Lenzing, H.~Meltzer, {\it The automorphism group of the derived category for a weighted projective line}, Comm. Algebra 28 (2000), no. 4, 1685--1700.
\bibitem{Meltzer} H.~Meltzer, {\it Exceptional vector bundles, tilting sheaves and tilting complexes for weighted projective lines}, Mem. Amer. Math. Soc. 171 (2004), no. 808.
\bibitem{OW} A.~Odesskii, T.~Wolf, {\it Compatible quadratic Poisson brackets related to a family of elliptic curves}, J. Geom. Phys. 63 (2013), 107--117.
\bibitem{Orlov} D.~Orlov, {\it Projective bundles, monoidal transformations, and derived categories of coherent sheaves}, Russian Acad. Sci. Izv. Math. 41 (1993), no. 1, 133--141.
\bibitem{P-Pois}  A.~Polishchuk,    {\it Algebraic geometry of Poisson brackets},
         Journal of Math. Sciences 84 (1997) 1413--1445.
\bibitem{P-ell}
A.~Polishchuk, {\it Poisson structures and birational morphisms associated with bundles on elliptic curves}, IMRN 13 (1998), 683--703.
\bibitem{Rains} E.~Rains, {\it Filtered deformations of elliptic algebras}, in {\it Hypergeometry, integrability and {L}ie theory}, Contemp. Math. 780, 95--154, 
Amer. Math. Soc., Providence, RI, 2022.
\bibitem{Silverman} Silverman, {\it Advanced Topics in the Arithmetic of Elliptic Curves}, Springer Verlag,  New York, 1994.
\end{thebibliography}
\end{document}